\documentclass[12pt]{amsart}
\usepackage[utf8]{inputenc}
\usepackage{amsfonts,amsmath,amssymb,amsthm}
\usepackage[plainpages=false]{hyperref}
\usepackage{esint}
\pdfmapline{=esint10 esint10 <esint10.pfb}
\usepackage{microtype}
\usepackage{cite}
\usepackage{yhmath}
\pdfmapline{=yrcmex10 Yhcmex <yhcmex.pfb}
\usepackage{graphicx}
\RequirePackage{color}

\numberwithin{equation}{section}

\newcommand{\mean}[1]{-\hskip-1.08em\int_{#1}}

\newtheorem{prop}{Proposition}[section]
\newtheorem{thm}[prop]{Theorem}
\newtheorem{lem}[prop]{Lemma}

\newtheorem{cor}[prop]{Corollary}
\newtheorem{rem}[prop]{Remark}

\newtheorem{defi}[prop]{Definition}

\newtheorem{ques}[prop]{Question}

\renewcommand{\det}{\operatorname{det}}

\allowdisplaybreaks
\title[Gradient potential for LMA equations]{Gradient potential estimates for linearized Monge-Amp\`ere equations}

\author{Guoqing Cui}

\date{\today}

\address{School of Mathematical Sciences, Peking University, Beijing 100871, China}
\email{gqcui25@stu.pku.edu.cn}

\thanks {This research is partially supported by 2023YFA009900 and NSFC  Grant 12271008.}

\begin{document}
	
	\subjclass{35B65, 35J96, 31C45.}
	
	\keywords{Gradient potential estimates, borderline-type regularity, linearized Monge-Amp\`ere equations}
	
	\begin{abstract}
    This paper establishes a pointwise gradient potential estimate for solutions to linearized Monge-Amp\`ere equations and derives a modulus of continuity estimate for the gradient in terms of the associated section adapted potential. As applications, we obtain a sufficient criterion for the continuity of the gradient, local BMO and VMO estimates for the gradient with respect to Monge-Amp\`ere sections, and a local $C^{1,\gamma}$ regularity criterion.
	\end{abstract}
	
	\maketitle
	
	\section{Introduction}
 In this paper, we investigate local gradient estimates for solutions of the linearized Monge-Amp\`ere equation 
 \begin{equation}\label{eq: LMA equation}
		\Phi^{ij}D_{ij}u=f\quad \text{in }\Omega,
	\end{equation} 
where $\phi\in C^2(\Omega)$ is strictly convex in a convex domain $\Omega\subset\mathbb{R}^n$, \[\Phi:=(\Phi^{ij})=\det D^2\phi (D^2\phi)^{-1}\] is the cofactor matrix of $D^2\phi$  satisfying 
\begin{equation}\label{condition: bdd of det}
		0<\lambda\le \det D^2\phi\le \Lambda\quad \text{in }\Omega,	\end{equation} 
for some positive constants $\lambda$ and $\Lambda$. 

The linearized Monge-Amp\`ere equation \eqref{eq: LMA equation} plays an important role in many problems. For example, it appears in a class of fourth-order equations of Monge-Amp\`ere type, including the affine mean curvature equation in affine differential geometry \cite{T, TW1, TW2, TW3, TW4} and Abreu's equation arising from complex geometry \cite{D1, D2, D3, D4, Z1, Z2}. In addition, the regularity of the dual semigeostrophic equation for atmospheric modeling is also related to the linearized Monge-Amp\`ere equation \cite{Lo, Le2}. A new application has recently been found in the study of convex functionals with a convexity constraint arising in the Rochet-Choné model for the monopolist problem in economics \cite{Le4, LZ, KLWZ}, which involves singular Abreu's equations. 

Under assumption \eqref{condition: bdd of det}, the linearized Monge-Amp\`ere equation \eqref{eq: LMA equation} fails to be uniformly elliptic in general, making the analysis challenging.
The first breakthrough in the regularity theory for the equation \eqref{eq: LMA equation} was made by Caffarelli and Guti\'errez. In \cite{CG}, they established the Harnack inequality for the homogeneous linearized Monge-Amp\`ere equation, which yields the interior H\"older continuity of solutions. Since then, many further developments have been obtained. H\"older estimates for the non-homogeneous equation have been extensively studied in \cite{Lo, LS, LN1, Le1, Le2, Le3, Wa, Ki, CWZ, GL}. In particular, when the right-hand side is a signed Radon measure, a potential estimate and interior H\"older estimates were obtained recently in \cite{CWZ, GL}.
	

Concerning gradient regularity, assuming \eqref{condition: bdd of det} together with the continuity of $\det D^2\phi$, Guti\'errez and Nguyen \cite{GN1} established the Cordes-Nirenberg type interior $C^{1,\alpha}$ estimates for solutions of the linearized Monge-Amp\`ere equation with  $f\in L^{q}(\Omega)$, $q>n$. When $f\in L^q(\Omega)$ with $n/2<q\le n$, the interior and global $W^{1,p}$ estimates were established in \cite{LN1}, namely, $u\in W^{1,p}(\Omega)$ for all $p<\frac{nq}{n-q}$. In the borderline case, $Du$ may fail to be continuous. Indeed, 
even in the model case $\phi=\frac{1}{2}|x|^2$, namely for the Poisson equation, $Du$ may be unbounded \cite{Ci}. 
In this paper, we establish the gradient potential estimate.


Recall that for the Poisson equation $\Delta u=f$, we have the well-known 
gradient potential estimate 
\[|Du(x)|\le C\mathbf{I}_1^{f}(x,r)+C\mean{B(x,r)}|Du|\,\mathrm{d}y\]
 for a.e. $x$, where $\mathbf{I}_1^{f}(x,r)$ 
 is the truncated Riesz potential of $f$ defined by 
 \[\mathbf{I}_1^f(x,r):=\frac{1}{\omega_n}\int_0^r\frac{|\mu|(B(x,\rho))}{\rho^{n-1}}\,\frac{\mathrm{d}\rho}{\rho}=\int_0^r\mean{B(x,\rho)}|f(y)|\,\mathrm{d}y\,\mathrm{d}\rho,\] 
with $\mu:=f\,\mathrm{d}x$, $\omega_n=|B_1(0)|$ and  \[\mean{B(x,r)}|Du|\,\mathrm{d}y:=\frac{1}{|B(x,r)|}\int_{B(x,r)}|Du|\,\mathrm{d}y.\]
 For fully nonlinear uniformly elliptic equations 
 \[F(x,D^2u)=f,\] 
 the following gradient potential estimate was established in \cite{DKM}:
 \[|Du(x)|\le C\mathbf{I}_p^f(x,r)+C\left(\mean{B(x,r)}|Du|^q\,\mathrm{d}x\right)^{1/q}\]
 for a.e. $x$, all $p>n_E$ and $q>n$, where $n_E$ is a constant depending only on $n$ and the ellipticity of $F$, and $\mathbf{I}_p^f(x,r)$ is the $L^p$ version of the classical Riesz potential defined by \[\mathbf{I}_p^f(x,r):=\int_0^r\left(\mean{B(x,\rho)}|f(y)|^p\,\mathrm{d}y\right)^{1/p}\,\mathrm{d}\rho\ge \mathbf{I}_1^f(x,r).\]
 
Because of the hidden nonlinear structure, degeneracy and affine invariance of the linearized Monge-Amp\`ere equation, we use the \textit{section-adapted weighted Riesz potential} (the $L^q$ version of the classical Riesz potential) defined by 
\begin{equation}\label{eq: Riesz potential}
		\mathbf{I}_{\phi, q}^{f,\alpha}(x_0,h):=\int_0^{h}\left(\frac{h}{\rho}\right)^{\alpha-\frac{1}{2}}\rho^{\frac{1}{2}}\left(\mean{S_{\phi}(x_0,\rho)}|f(y)|^q\,\mathrm{d}y\right)^{1/q}\frac{\mathrm{d\rho}}{\rho},
	\end{equation}
	for any $q\ge 1$, $\alpha\ge \frac{1}{2}$, where \[S_{\phi}(x_0,h):=\{x\in\Omega: \phi(x)<\phi(x_0)+D\phi(x_0)\cdot (x-x_0)+h\}\] is a section of $\phi\in C^2(\Omega)$ centered at $x_0\in\Omega$ with height $h>0$. Under the assumption \eqref{condition: bdd of det}, sections are equivalent to Euclidean balls in terms of their geometric and measure-theoretic properties; see \cite{Le5} for an overview of these properties.

For convenience, we denote 
\[\det D^2\phi:=g\in C(\Omega).\] 
We do not address the question of optimal regularity assumptions on $g$. Throughout the paper, we assume that $\phi$ is smooth. Nevertheless, all estimates depend only on the structural constants appearing in the assumptions, and are independent of smoothness assumption on $\phi$, and upper or lower bounds for the eigenvalues of $D^2\phi$.
We have the following pointwise gradient potential estimate:
	\begin{thm}[Gradient potential estimate]\label{thm: gradient potential est}
	Let $\Omega$ be a normalized convex domain in $\mathbb{R}^n$,
	i.e. $B_1(0)\subset\Omega\subset B_n(0)$, and let $\phi\in C^2(\Omega)\cap C(\overline{\Omega})$ be a strictly convex function satisfying \begin{equation}\label{eq: continuous g}
			\phi=0\: \text{on }\partial\Omega,\quad0<\lambda\le g\le\Lambda\,\,\text{in }\Omega,\quad\text{and}\quad g\in C(\Omega).
		\end{equation} Let $f\in L^q(\Omega)$ with  $\frac{n}{2}<q\le n$, and let $u$ solve \eqref{eq: LMA equation}. Choose $p$ so that \begin{equation}\label{eq: p-choice}
		n<p<\frac{nq}{n-q}\quad \text{if  }q<n,\quad\quad\text{and }\quad\quad  n<p<\infty\quad \text{if  }q=n.
	\end{equation}
		For any $\theta\in(0,1)$ and $\alpha\in(1/2,1)$, for any Lebesgue point $x_0\in \Omega'$ of $Du$ with $\Omega'\Subset\Omega$, we have 
		\begin{equation}\label{eq: gredient potential est}
			\begin{aligned}
			|Du(x_0)|\le& C\mean{S_{\phi}(x_0,h)}|Du(x)|\,\mathrm{d}x+ C\left(\mean{S_{\phi}(x_0,h^{\theta})}|Du(x)|^p\,\mathrm{d}x\right)^{1/p}\\[5pt]
			&\qquad + C \mathbf{I}_{\phi,q}^{f,\alpha}(x_0,ch^{\theta}),
			\end{aligned}
		\end{equation}
		whenever $0<h\le h_{\ast}$ and $S_{\phi}(x_0,ch^{\theta})\Subset\Omega$,
		 where the constants $C$, $c$, $h_{\ast}>0$ depend only on $n$, $\lambda$, $\Lambda$, $p$, $q$, $\alpha$, $\theta$,  $\operatorname{dist}(\Omega',\partial \Omega)$ and the modulus of continuity of $g$.
	\end{thm}

  \begin{rem} We give several remarks on the modified Riesz potential \eqref{eq: Riesz potential}:
 	\begin{enumerate}
 		\item [(1)] The form of \eqref{eq: Riesz potential} is dictated by the iteration in Section \ref{sec: gradient potential est}. At height $\rho$, the natural contribution of $f$ to the gradient estimate is \[\rho^{\frac{1}{2}}\left(\mean{S_{\phi}(x_0,\rho)}|f(y)|^q\,\mathrm{d}y\right)^{1/q}.\] To compensate for the growth of the affine normalization matrices, we introduce the additional weight $\left(\frac{h}{\rho}\right)^{\alpha-\frac{1}{2}}$, as reflected in \eqref{eq: sum-E_i-est}.
 		In particular, when $\phi=\frac{1}{2}|x|^2$ and $\alpha=\frac{1}{2}$, we have 
 		\begin{align*}
 			\mathbf{I}_{\phi, q}^{f,\alpha}(x_0,h)=\sqrt{2}\mathbf{I}_q^f(x_0,\sqrt{2h}),
 		\end{align*} where $\mathbf{I}_{q}^f(x_0,h)$ is the usual $L^q$-version Riesz potential defined in \cite{DKM}.
 		\item [(2)] We can also define the section-adapted Wolff potential as \begin{equation}\label{eq: Wolff potential}
 			W_{\phi,\beta,q}^{\mu}(x,h):=\int_0^h\left(\frac{|\mu|(S_{\phi}(x,\rho))}{\rho^{\frac{1}{2} \left(n-\beta q\right)}}\right)^{\frac{1}{q-1}}\frac{\mathrm{d}\rho}{\rho},
 		\end{equation} for $0<\beta\le n/q$. When $\phi=\frac{1}{2}|x|^2$, we can derive \[W_{\phi,\beta,q}^{\mu}(x,h)=C_{n,\beta, q}W_{\beta,q}^{\mu}(x,\sqrt{2h}),\] where $W_{\beta,q}^{\mu}(x,r)$ is the usual Wolff potential.
 		Then we can write \eqref{eq: Riesz potential} as \begin{equation}\label{eq: use W to I}
 			I_{\phi,q}^{f,\alpha}(x_0,h)\approx h^{\alpha-\frac{1}{2}}W^{|f|^q}_{\phi,\beta,q+1}(x_0,h)\qquad \beta:=\frac{2(1-\alpha) q}{q+1}.
 		\end{equation}
 	\end{enumerate}
 \end{rem}

    \begin{rem}
       We give several remarks on the gradient potential estimate \eqref{eq: gredient potential est}:
       \begin{enumerate}
           \item   The modified Riesz potential $\mathbf{I}_{\phi,q}^{f,\alpha}(x_0,h)$ can be seen as an integrated version of the quantity $N_{\phi,f,q}(x_0)$ defined in \cite{LN1} by \[N_{\phi,f,q}(x_0):=\sup_{h\le h_{0}} h^{\frac{1-\varepsilon}{2}}\left(\mean{S_{\phi}(x_0,h)}|f(x)|^q\,\mathrm{d}x\right)^{1/q}\] for $\varepsilon\in (0,1)$ and $h_0>0$. Let $\alpha\in (\frac{1}{2}, \frac{1+\varepsilon}{2})$, we derive \begin{align*}
            \mathbf{I}_{\phi,q}^{f,\alpha}(x_0,h_0)&=h_0^{\alpha-\frac{1}{2}}\int_0^{h_0}\left(\mean{S_{\phi}(x_0,\rho)}|f(x)|^q\,\mathrm{d}x\right)^{1/q}\frac{\mathrm{d}\rho}{\rho^{\alpha}}\\
            &\le \left(\sup_{\rho\le h_0}\rho^{\frac{1-\varepsilon}{2}}\left(\mean{S_{\phi}(x_0,\rho)}|f(x)|^q\,\mathrm{d}x\right)^{1/q}\right) h_0^{\alpha-\frac{1}{2}}\int_0^{h_0}\rho^{-\frac{1-\varepsilon}{2}}\frac{\mathrm{d}\rho}{\rho^{\alpha}}\\
            &\le C_{\varepsilon,\alpha}h_0^{\varepsilon/2}N_{\phi,f,q}(x_0).
        \end{align*}
       Consequently, after estimating the local averages of \(Du\) in
      Theorem~\ref{thm: gradient potential est} by the standard local \(W^{1,p}\) estimate, one recovers
     a pointwise gradient bound of the same form as that in \cite{LN1},
     namely \[ |Du(x_0)| \le C\bigl(\|u\|_{L^\infty(\Omega)}+N_{\phi,f,q}(x_0)\bigr).\]
     Therefore Theorem~\ref{thm: gradient potential est} should be viewed as complementary to the pointwise estimate of \cite{LN1} in a different form.
     \item The gradient potential estimate \eqref{eq: gredient potential est} can yield directly the following gradient estimate in Lorentz spaces (see Definition \ref{defn: Lorentz and Morrey space}):  \begin{equation}\label{eq: Lorentz est}
            \Vert Du\Vert_{L^{r,\gamma}(\Omega')}\le C\left(\Vert u\Vert_{L^{\infty}(\Omega)}+\Vert f\Vert_{L^{s,\gamma}(\Omega)}\right),
        \end{equation} for any $s\in (n/2,n]$, $r\in (n, ns/(n-s))$ and $\gamma\in (0,\infty]$.  The gradient estimate \eqref{eq: Lorentz est} can be viewed as an analogue of the classical $W^{1,p}$ estimates of \cite{LN1} in Lorentz spaces. One way to justify \eqref{eq: Lorentz est} is to combine \eqref{eq: gredient potential est}, \eqref{eq: use W to I}, the Havin-Mazya type inequality \[W_{\phi,\beta,q}^{\mu}(x,h)\le C\mathbf{I}_{\phi,\beta}\left[(\mathbf{I}_{\phi,\beta}(|\mu|))^{\frac{1}{q-1}}\right](x,Ch),\] and the Lorentz estimate for $\mathbf{I}_{\phi,\beta}$ \[\Vert \mathbf{I}_{\phi,\beta}(f)(\cdot, h_0)\Vert_{L^{\frac{ns}{n-\beta s},\gamma}(\Omega')}\le C\Vert f\Vert_{L^{s,\gamma}(\Omega)}.\] 
        Equivalently, one may obtain \eqref{eq: Lorentz est} more directly from the $W^{1,p}$-estimate of \cite{LN1} together with standard Lorentz space properties. Therefore, \eqref{eq: Lorentz est} is recorded here only as a consequence and not as the main new contribution.
       \end{enumerate}
    \end{rem}

The potential estimate also yields a modulus of continuity for $Du$:
\begin{thm}[Gradient continuity via potential]\label{thm: grad-conti}
		Under the assumptions of Theorem \ref{thm: gradient potential est}, assume that for every $\Omega'\Subset\Omega$, \begin{equation}
			\lim\limits_{h\to 0} \sup_{\substack{x\in\Omega',\\S_{\phi}(x,h)\Subset\Omega}}\mathbf{I}_{\phi,q}^{f,\alpha}(x,h)=0.
		\end{equation} Then $u\in C_{\text{loc}}^1(\Omega)$.  Moreover, for every $\Omega'\Subset\Omega''\Subset\Omega$ and every $\theta\in (0,1)$,  we have \begin{equation} \label{eq: Du-modulus}
		|Du(x_1)-Du(x_2)|\le C\Vert Du\Vert_{L^{\infty}(\Omega'')}|x_1-x_2|^{\kappa\beta} +C\max_{x\in \{x_1,x_2\}}\mathbf{I}_{\phi,q}^{f,\alpha}(x,c|x_1-x_2|^{\theta\beta}) 
		\end{equation} for all $x_1,x_2\in\Omega'$. Here the constants $C$, $c>0$ depend only on $n$, $\lambda$, $\Lambda$, $p$, $q$, $\alpha$, $\theta$, $\operatorname{dist}(\Omega',\partial\Omega'')$, $\operatorname{dist}(\Omega'',\partial\Omega)$ and the modulus of continuity of $g$. The exponents $\kappa>0$ and $\beta>1$ depend only on $n$, $p$, $q$, $\alpha$, $\theta$ and $n$, $\lambda$, $\Lambda$, respectively.
\end{thm}

	
We next establish several borderline-type regularity results as consequences of the potential estimate. These should be regarded as sufficient conditions formulated in the language of section geometry. We make no claim regarding their optimality. Related borderline regularity theories have been developed for $p$-Laplacian type equations \cite{DM1,DM2,DZ,Mi2} and fully nonlinear equations \cite{DKM}.

	\begin{cor}\label{cor: borderline regularity}
		Let $\Omega$ and $\phi$ satisfy \eqref{eq: continuous g} with $\Omega$ normalized, and let $u$ be a solution of \eqref{eq: LMA equation}. Fix $q\in (n/2,n]$ and $\alpha\in (1/2,1)$. For $\Omega'\Subset\Omega$, set \[\omega_{\alpha,q}(r,\Omega'):=\sup_{\substack{x\in\Omega',0<h<r\\S_{\phi}(x,h)\Subset\Omega}}\left(h^{q(1-\alpha)-\frac{n}{2}}\int_{S_{\phi}(x,h)}|f(y)|^q\,\mathrm{d}y\right)^{1/q}.\]  Then \begin{enumerate}
				\item [(1)] if $\omega_{\alpha,q}(r,\Omega')\to 0$ as $r\to 0$ for every $\Omega'\Subset\Omega$, then $Du\in \text{VMO}_{\text{loc}}(\Omega,\phi)$ (see Definition \ref{defn: BMO VMO});
				\item [(2)] if \[\sup_{0<r\le r_0}\omega_{\alpha,q}(r,\Omega')<\infty\] for every $\Omega'\Subset\Omega$ and some $r_0>0$, then $Du\in \text{BMO}_{\text{loc}}(\Omega,\phi)$ (see Definition \ref{defn: BMO VMO});
				\item [(3)] if for some $\varepsilon\in (0,1)$, \[\sup_{S_{\phi}(x,h)\Subset\Omega}h^{\frac{1-\varepsilon}{2}}\left(\mean{S_{\phi}(x,h)}|f(y)|^n\,\mathrm{d}y\right)^{1/n}<\infty,\] then $u\in C_{\text{loc}}^{1,\gamma}(\Omega)$ for some $\gamma>0$.
			\end{enumerate}
		In particular, \begin{enumerate}
		    \item [(1)] if $f\in L^{s,\infty}$ locally in $\Omega$ for some $s>n$, then $Du\in BMO_{\text{loc}}(\Omega,\phi)$; 
            \item [(2)] if $f\in L^{r}(\Omega)$ for some $r>n$, then $u\in C_{\text{loc}}^{1,\gamma}(\Omega)$ for some $\gamma>0$.
		\end{enumerate} 
	\end{cor}
  
 \begin{rem}
   The conditions in Corollary~\ref{cor: borderline regularity} are sufficient conditions expressed in terms of sections. They are not asserted to be sharp. The endpoint cases, such as the endpoint in the Lorentz estimate, require additional arguments and are left open in Question~\ref{ques:endpoint}.
   \end{rem}
   
   The restriction $\theta<1$ is essential to our proof, as it allows us to ensure that $L^{1+\theta}\rho^{1-\theta}<1$  in \eqref{eq: choice of sigma and delta}, by choosing the oscillation of $g$ sufficiently small so that $L$ is close to one. The iteration from $N=\left\lceil\theta k\right\rceil+1$ to $k$ compensates for the growth of the normalization matrices, but requires a larger section of height comparable to $h^{\theta}$. 
   In short, the lack of uniform ellipticity for \eqref{eq: LMA equation} under \eqref{eq: continuous g}  makes it particularly delicate to recover the natural scale in the gradient potential estimate and the critical Sobolev exponent in the Lorentz estimate. This leads us to the following interesting question.
    
    \begin{ques}\label{ques:endpoint}
        Compared with uniformly elliptic equations, the present approach loses a small amount of scale because of the degeneracy and affine normalization of \eqref{eq: LMA equation}. Therefore, two natural questions arise:
        \begin{enumerate}
        	\item  Does Theorem~\ref{thm: gradient potential est} remain valid when $\theta=1$?
        	\item  Can the critical exponent
        	$r=\frac{ns}{n-s}$, with $n/2<s<n$, be attained in
        	\eqref{eq: Lorentz est}?
        \end{enumerate} 
    \end{ques}
    
	The paper is organized as follows. In Section \ref{sec: pre}, we collect some useful lemmas that will be used later. We prove the gradient potential estimate (Theorem~\ref{thm: gradient potential est}) in Section \ref{sec: gradient potential est} and derive gradient-continuity estimates (Theorem~\ref{thm: grad-conti}) and borderline-type regularity results (Corollary~\ref{cor: borderline regularity}) in Section \ref{sec: gradient continuity}. 

	\section{Preliminaries}\label{sec: pre}
	\subsection{Geometry of sections}
	
	For a strictly convex function $\phi$ defined in $\Omega$ and $h>0$, we denote by $S_h(\phi)$ the section of $\phi$ centered at its minimum point $y_0$ with height $h$, i.e. 
	\[S_h(\phi):=\left\{x\in\Omega: \phi(x)< \min_{\Omega}\phi+h\right\}=S_{\phi}(y_0,h).\]
	\begin{lem}\label{lem: geo of sections}
		\textnormal{\cite[Lemma 3.2]{GN1}}  Suppose $B_1(0)\subset \Omega\subset B_n(0)$ is a normalized convex domain. Then there exist universal constants $h_0, \tau_0>0$ and a positive definite matrix $M=A^tA$ satisfying 
		\[\det M=1,\quad\text{and }\quad 0<c_1I\le M\le c_2I,\]
		 such that if $\phi\in C^2(\Omega)$ is a strictly convex function in $\Omega$ with \[1-\varepsilon\le \det D^2\phi\le 1+\varepsilon\,\, \text{in }\Omega, \quad\text{and }\quad\phi=0 \text{ on }\partial\Omega,\] then for $0<h\le h_0$ and $\varepsilon\le \tau_0h^2$, we have \[B_{(1-C(h^{1/2}+h^{-1}\varepsilon^{1/2}))\sqrt{2}}(0)\subset h^{-1/2}TS_{h}(\phi)\subset B_{(1+C(h^{1/2}+h^{-1}\varepsilon^{1/2}))\sqrt{2}}(0),\] 
		where $Tx:=A(x-x_0)$ and $x_0\in\Omega$ is the minimum point of $\phi$.
	\end{lem}
	
	\begin{lem}\label{lem: geo-2}
			\textnormal{\cite[Lemma 3.3]{GN1}} Suppose $\Omega$ is a convex domain satisfying
			\[B_{(1-\sigma)\sqrt{2}}(0)\subset \Omega\subset B_{(1+\sigma)\sqrt{2}}(0)\]  with $0< \sigma\le \frac{1}{4}$. There exist $\mu_0, \tau_0>0$ independent of $\sigma$, and a positive definite matrix $M=A^tA$, with \[\det M=1\quad\text{and }\quad (1-C\sigma)I\le M\le (1+C\sigma)I,\] 
			such that if $\phi\in C^2(\Omega)$ is a strictly convex function in $\Omega$ satisfying 
			\[1-\varepsilon\le\det D^2\phi\le 1+\varepsilon\,\,\,\text{in }\Omega,\quad \text{and} \quad \phi=0 \,\,\,\text{on }\partial\Omega,\] then for $\mu\in(0,\mu_0)$ and $\varepsilon\le\tau_0\mu^2$ we have \[B_{(1-C(\sigma\mu^{1/2}+\mu^{-1}\varepsilon^{1/2}))\sqrt{2}}(0)\subset\mu^{-1/2}TS_{\mu}(\phi)\subset B_{(1+C(\sigma\mu^{1/2}+\mu^{-1}\varepsilon^{1/2}))\sqrt{2}}(0),\] where $Tx:=A(x-x_0)$ and $x_0\in\Omega$ is the minimum point of $\phi$.
	\end{lem}
	\begin{lem}\label{lem: average contin}
			Let $\phi\in C^2(\Omega)$ be a strictly convex function in $\Omega$. For any $f\in L^{\infty}(\Omega)$ and any fixed $h>0$, we define \[F(x)=(f)_{S_{\phi}(x,h)}:=\frac{1}{|S_{\phi}(x,h)|}\int_{S_{\phi}(x,h)}f(y)\,\mathrm{d}y\] whenever $ S_{\phi}(x,2h)\Subset\Omega$. Then $F$ is continuous.
	\end{lem}
	\begin{proof}
		Let $x_0\in\Omega$ and choose $h>0$ such that
		$S_\phi(x_0,2h)\Subset\Omega$.
		Set $M:=\|f\|_{L^\infty(\Omega)}$.
		If $M=0$, then $F\equiv 0$, and the conclusion is trivial. Hence we assume $M>0$. We define the function \[\Psi(x,y):=\phi(y)-\phi(x)-D\phi(x)\cdot (y-x).\]
		Then \[S_\phi(x,h)=\{y\in\Omega:\Psi(x,y)<h\},\]
		and the function $\Psi$ is continuous in $(x,y)$.
		
		We first prove that
		\[
		|S_\phi(x,h)\triangle S_\phi(x_0,h)|\to 0
		\qquad\text{as }x\to x_0.
		\]
		Let $0<t<h$. We claim that there exists $\delta_t>0$ such that if $|x-x_0|<\delta_t$, then
		\[S_\phi(x_0,h-t)\subset S_\phi(x,h)\subset S_\phi(x_0,h+t).\]
		Indeed, since $\Psi$ is continuous, we may choose $\delta_t>0$ small enough such that
		\[
		|\Psi(x,y)-\Psi(x_0,y)|<t
		\]
		for all $y\in \overline{S_\phi(x_0,h-t)}$ whenever $|x-x_0|<\delta_t$.
		Thus, if $y\in S_\phi(x_0,h-t)$, then
		\[
		\Psi(x,y)
		\le \Psi(x_0,y)+t
		< h.
		\]
		Hence $y\in S_\phi(x,h)$,
		and therefore $S_\phi(x_0,h-t)\subset S_\phi(x,h)$.
		
		It remains to prove the second inclusion. Let $U:=S_\phi(x_0,h+t)$.  For every $z\in\partial U$, we have $\Psi(x_0,z)=h+t$. Using the continuity of $\Psi$ and the compactness of $\partial U$, and taking
		$\delta_t>0$ smaller if necessary, we may assume that
		\[
		\Psi(x,z)>h
		\qquad\text{for all }z\in\partial U
		\]
		whenever $|x-x_0|<\delta_t$. We may also assume that $x\in U$, since
		$x_0\in U$ and $U$ is open.
	    Suppose, by contradiction, that there exists
		\[y\in S_\phi(x,h)\setminus S_\phi(x_0,h+t).\]
	     Since $x\in U$, there exists $z_0\in [x,y]\cap\partial U.$
	    The convexity of $\Psi$ in $y$ yields $\Psi(x,z_0)<h$,
	  which contradicts the assumption. Therefore $S_\phi(x,h)\subset S_\phi(x_0,h+t)$.
		
		Consequently, for $|x-x_0|<\delta_t$, we have
		\[
		S_\phi(x,h)\triangle S_\phi(x_0,h)
		\subset
		S_\phi(x_0,h+t)\setminus S_\phi(x_0,h-t).
		\]
	The convexity of sections yields $|\partial S_\phi(x_0,h)|=0$. Therefore
		\[|S_\phi(x_0,h+t)\setminus S_\phi(x_0,h-t)|\to 0.\]
		It follows that
		\[|S_\phi(x,h)\triangle S_\phi(x_0,h)|\to 0
		\qquad\text{as }x\to x_0.\]
		In particular,
		\[\big||S_\phi(x,h)|-|S_\phi(x_0,h)|\big|
		\le
		|S_\phi(x,h)\triangle S_\phi(x_0,h)|\to 0.\]
	Thus, for $x$ sufficiently close to $x_0$, we have
		\[
		|S_\phi(x,h)|\ge \frac12 |S_\phi(x_0,h)|.
		\]
		
		Now set
		\[
		S_x:=S_\phi(x,h),
		\qquad
		S_0:=S_\phi(x_0,h).
		\]
		Hence, for $x$ sufficiently close to $x_0$, we have \begin{align*}
			|F(x)-F(x_0)|&=\left|\frac1{|S_x|}\int_{S_x}f(y)\,\mathrm{d}y-\frac1{|S_0|}\int_{S_0}f(y)\,\mathrm{d}y\right| \\
		&\le\frac1{|S_x|}\left|\int_{S_x}f(y)\,\mathrm{d}y-\int_{S_0}f(y)\,\mathrm{d}y\right|+\left|\frac1{|S_x|}-\frac1{|S_0|}\right|\int_{S_0}|f(y)|\,\mathrm{d}y\\
		&\le \frac{2M}{|S_0|}|S_x\triangle S_0|+\frac{2M}{|S_0|}|S_x\triangle S_0|=\frac{4M}{|S_0|}|S_x\triangle S_0|\to 0,
		\end{align*} as $x\to x_0$. Therefore, $F$ is continuous.
	\end{proof}
	\begin{rem}
		The continuity property in Lemma \ref{lem: average contin} does not essentially depend on the smoothness of $\phi$. If $\phi$ is merely convex, we may use a chosen subgradient. Namely, for $p\in\partial\phi(x)$, set
		\[
		S_\phi(x,p,h):=
		\{y\in\Omega:\phi(y)<\phi(x)+p\cdot(y-x)+h\}.
		\]
		If
		\[
		x_k\to x_0,\quad p_k\in\partial\phi(x_k),\quad p_k\to p_0\in\partial\phi(x_0),\quad\text{  and  }\quad
		S_\phi(x_0,p_0,2h)\Subset\Omega,
		\]
		then
		\[
		|S_\phi(x_k,p_k,h)\triangle S_\phi(x_0,p_0,h)|\to0.
		\]
		Consequently, for every $f\in L^\infty(\Omega)$,
		\[
		\frac1{|S_\phi(x_k,p_k,h)|}
		\int_{S_\phi(x_k,p_k,h)}f(y)\,\mathrm{d}y
		\to
		\frac1{|S_\phi(x_0,p_0,h)|}
		\int_{S_\phi(x_0,p_0,h)}f(y)\,\mathrm{d}y .
		\]
		Thus, without differentiability, the natural continuity statement holds with
		respect to the pair $(x,p)$, where $p\in\partial\phi(x)$, rather than $x$ alone.
	\end{rem}
	
\subsection{Lorentz spaces}
	\begin{defi}\label{defn: Lorentz and Morrey space}
		For any $0<p,q\leq \infty$, we define 
		\[\Vert f\Vert_{L^{p, q}(\Omega)}:=\left\{\begin{aligned}
			&\left(\int_0^{\infty}\left(t^{\frac{1}{p}}f^{\ast}(t)\right)^q\frac{\mathrm{d}t}{t}\right)^{1/q}, && q<\infty, \\[7pt]
			&\sup_{t>0} t^{\frac{1}{p}}f^{\ast}(t),&& q=\infty,
		\end{aligned}\right.\] where \begin{equation*}\label{eq: rearrangement of a function}
		    f^{\ast}(t):=\inf\{s>0: |\{x\in\Omega: |f(x)|>s\}|\le t\}.
		\end{equation*} The Lorentz spaces are defined by $$L^{p,q}(\Omega):=\{f:\Omega\to\mathbb{R} \text{ measurable}: \Vert f\Vert_{L^{p,q}(\Omega)}<\infty\}.$$
	\end{defi}
\begin{lem}\label{lem: Holder eq}
		\textnormal{(H\"older-type inequality)} For $1\leq p,  p_1, p_2<\infty$ and $1<q, q_1, q_2\leq\infty$ satisfying 
	\[\frac{1}{p_1}+\frac{1}{p_2}=\frac{1}{p}, \ \ \frac{1}{q_1}+\frac{1}{q_2}=\frac{1}{q},\] there exists $C_{p_1, p_2, q_1, q_2}>0$ such that 
	\begin{equation}
		\Vert fg\Vert_{L^{p,q}(\Omega)}\le C_{p_1, p_2, q_1, q_2}\Vert f\Vert_{L^{p_1,q_1}(\Omega)}\cdot\Vert g\Vert_{L^{p_2,q_2}(\Omega)}.\label{property 3: Holder ineq}
	\end{equation}
\end{lem}

	\begin{defi}\label{defn: BMO VMO}
		\textnormal{\cite{H}} For $f\in L^1(\Omega)$ and $A\subset\Omega$, the mean oscillation of $f$ over $A$ is defined by \[\operatorname{mosc}_A f:=\mean{A}|f(x)-(f)_A|\,\mathrm
		dx,\quad \text{ where }(f)_A:=\mean{A}f(y)\,\mathrm{d}y.\]
		Let $\phi\in C^2(\Omega)$ be a strictly convex function on $\Omega$ satisfying \eqref{condition: bdd of det}. For any $\Omega'\Subset\Omega$, we write \[Q_f(r,\Omega'):=\sup_{\substack{x_0\in\Omega'\\\operatorname{diam}(S_{\phi}(x_0,h))\le r}}\operatorname{mosc}_{S_{\phi}(x_0,h)} f.\] We say  $f\in BMO_{\text{loc}}(\Omega,\phi)$
	and	 $f\in VMO_{\text{loc}}(\Omega,\phi)$ if, respectively, 
		\[\displaystyle\limsup_{r\to0}Q_f(r,\Omega')<\infty \text{\ \ and \ \ } \lim\limits_{r\to 0}Q_f(r,\Omega')=0\]
		for any $\Omega'\Subset\Omega$. 
	\end{defi}
   
	
	\subsection{Interior \texorpdfstring{$W^{1,p}$}{W1p} estimates}

	\begin{lem}\label{lem: W1p-est}
	Let $\phi\in C^2(\Omega)$ be a convex function in a convex domain $\Omega\subset\mathbb{R}^n$ with normalized $S_1:=S_{\phi}(0,1)\Subset\Omega$ and $f\in L^q(\Omega)$ with $\frac{n}{2}<q\le n$. Choose $p$ so that \[n<p<\frac{nq}{n-q}\quad \text{if  }q<n,\quad\quad\text{and }\quad\quad  n<p<\infty\quad \text{if  }q=n.\]
	Suppose $u$ is a solution of \eqref{eq: LMA equation} in $\Omega$. Then there exists $\delta\in (0,1)$ depending only on $n$, $p$ and $q$ such that if \[\begin{cases}
			1-\delta\le \det D^2\phi\le 1+\delta, & \text{in }S_1,\\
		 			\phi=0 & \text{on }\partial S_1,\end{cases}\]
		 then \[\left(\mean{S_{1/2}}|Du|^p\,\mathrm{d}x\right)^{1/p}\le C\Vert u\Vert_{L^{\infty}(S_1)}+C \left(\mean{S_1}|f|^{q}\,\mathrm{d}x\right)^{1/q},\] where $C=C(n,p,q)$.
	\end{lem}
	\begin{proof}
		The proof is similar to \cite[Theorem 1.3]{LN1}, which is a direct combination of \cite[Theorem 1.2]{LN1} and \cite[Theorem 2.2]{GN2}.
	\end{proof}
	
	\subsection{\texorpdfstring{$C^{1,1}$}{C1alpha} estimates for uniformly elliptic equations} 
	The following form of the standard interior estimate is used only on subdomains separated from the boundary.
	It follows from Pogorelov-type estimates for the solution $w$ of \eqref{eq: detD^2w=1} and the classical interior $C^{1,1}$ theory for the resulting uniformly elliptic linear equation; see \cite[Lemma 3.3]{LN1} and \cite{Gu}.

	\begin{lem}\label{lem: c1alpha for hom data}
		Let $\Omega\subset\mathbb{R}^n$ be a normalized convex domain with $B_1(0)\subset \Omega\subset B_n(0)$ and let $w\in C^{4}(\Omega)\cap C(\overline{\Omega})$ be a convex solution of \begin{equation}\label{eq: detD^2w=1}
			\begin{cases}
				\det D^2w=1 & \text{in }\Omega\\
				w=0 & \text{on }\partial\Omega.	
			\end{cases}
		\end{equation} 
		\begin{enumerate}
			\item [(i)] Let $x_1\in\Omega$ be the minimum point of $w$. Then $|w(x_1)|\sim c_n$ for some universal constant $c_n$ and we have the Pogorelov's estimates: \[\frac{2}{C_2^2}I\le D^2w\le \frac{2}{C_1^2}I\, \text{ for all }x\in\Omega\text{ with }\operatorname{dist}(x,\partial\Omega)\ge c_n,\] where $C_1$ and $C_2$ are constants depending only on $n$.
			\item [(ii)] For any solution $v\in C^2(B_1)$ of $W^{ij}D_{ij}v=0\,\,\text{in }B_1(0)$, where $W=(W^{ij}):=\det D^2w(D^2w)^{-1}$, we have the classical interior $C^{1,1}$ estimate\[\Vert v\Vert_{C^{1,1}(B_{1/2}(0))}\le c\Vert v\Vert_{L^{\infty}(\partial B_{3/4}(0))}\] where $c$ is a constant depending only on $n$.
		\end{enumerate}
	\end{lem}

\subsection{An approximation lemma} The following approximation is a direct consequence of \cite{GN1, LN1}.
\begin{lem}\label{lem: approx}
	Let $\phi\in C^2(\Omega)$ be a strictly convex function in a convex domain $\Omega\subset\mathbb{R}^n$. Assume that $S_1:=S_{\phi}(0,1)\Subset\Omega$, $S_1$ is normalized, i.e. $B_1(0)\subset S_1\subset B_n(0)$, and $0$ is the minimum point of $\phi$. Let $f\in L^q(\Omega)$ with $\frac{n}{2}<q\le n$ and $u$ be a solution of \eqref{eq: LMA equation} with $\Vert u\Vert_{L^{\infty}(S_1)}\le M$ for some constant $M$. Then, given $\varepsilon>0$, there exists $\delta=\delta(\varepsilon,n,q,M)>0$ such that if \[\begin{cases}
1-\delta\le \det D^2\phi\le 1+\delta &\text{in }S_1,\\
\phi=0 & \text{on }\partial S_1,
\end{cases}\quad \text{and }\quad \left(\mean{S_1}|f(x)|^q\,\mathrm{d}x\right)^{1/q}\le \delta,\] then the solution $v$ of \[\begin{cases}
W^{ij}D_{ij}v=0 & \text{in }S_{1/2},\\
v=u & \text{on } \partial S_{1/2},
\end{cases} \quad \text{where }\quad \begin{cases}
\det D^2w=1 & \text{in }S_1,\\
w=0& \text{on }\partial S_1,
\end{cases}\] satisfies \[\Vert u-v\Vert _{L^{\infty}(S_{1/4})}<\varepsilon,\] where $ W$ is the cofactor matrix of  $D^2w$.
\end{lem}
	\begin{proof}
		This is an extension of \cite[Lemma 4.2]{GN1} from $q=n$ to $q\in (n/2,n]$. Since $\phi=0$ on $\partial S_1$ and $\min_{S_1}\phi=-1$, the sections $S_{1/2}$ and $S_{1/4}$ coincide with $U_{1/2}$ and $U_{1/4}$ in \cite[Lemma 2.5]{LN1}, with $U=S_1$. By the interior cofactor stability estimate in \cite[Lemma 4.4]{GN2} and H\"older inequality, \[\lim\limits_{\delta\to 0}\Vert\Phi-W\Vert_{L^q(S_{1/2})}=0\] uniformly under the stated assumptions, since $q\le n$ and $S_1\subset B_n(0)$. Moreover, $\Vert f\Vert_{L^q(S_1)}\le |B_n(0)|^{\frac{1}{q}}\delta$.
		Applying \cite[Lemma 2.5]{LN1} to $u/(M+1)$ and $v/(M+1)$ with $\alpha_1=\frac{1}{2}$ and $\alpha_2=\frac{1}{4}$, we obtain \[\Vert u-v\Vert_{L^{\infty}(S_{1/4})}\le C\Vert\Phi-W\Vert_{L^q(S_{1/2})}^{\gamma}+C\Vert f\Vert_{L^q(S_1)},\]
		where $C=C(n,q,M)>0$ and $\gamma=\gamma(n,q)\in(0,1)$. Taking $\delta=\delta(\varepsilon,n,q,M)$ sufficiently small, so that the smallness condition in \cite[Lemma 2.5]{LN1} is also satisfied, we obtain the desired estimate.
	\end{proof}

	\section{The gradient potential estimates}\label{sec: gradient potential est}
	
	For the rest of the paper, we consider the excess functional \[E_p(S):=\left(\mean{S}|Du-(Du)_S|^p\,\mathrm{d}x\right)^{1/p},\] where \[(Du)_S:=\mean{S} Du(x)\,\mathrm{d}x,\] and $S\subset\Omega$ denotes a section with positive height and $p\ge 1$, and $u$ is a solution of \eqref{eq: LMA equation}. We have the following crucial decay estimate: 
	\begin{lem}\label{lem: decay-est}
		Let $\phi\in C^2(\Omega)$ be a strictly convex function in a  convex domain $\Omega\subset\mathbb{R}^n$ with normalized $S_1:=S_{\phi}(0,1)\Subset\Omega$, where $0$ is the minimum point of $\phi$, and $f\in L^q(\Omega)$ with $\frac{n}{2}<q\le n$. Let $u$ be a solution of \eqref{eq: LMA equation}. Given $p$ with $n<p<\frac{nq}{n-q}$ if $q<n$, and $n<p<\infty$ if $q=n$, and $c_0\in (0,1)$, there exist $\delta,\sigma\in (0,1)$, both depending only on $n$, $p$, $q$ and $c_0$ such that if 
		\[E_p(S_1)\le 1,\quad \left(\mean{S_1}|f(x)|^q\,\mathrm{d}x\right)^{1/q}\le \delta\]
		and \[ \begin{cases}
		1-\delta\le \det D^2\phi\le 1+\delta, & \text{in }S_1,\\
		\phi=0 & \text{on }\partial S_1,
		\end{cases}\] then \[E_p(S_{\sigma})\le c_0,\] where $S_{\sigma}:=S_{\phi}(0,\sigma)$.
	\end{lem}
	\begin{proof}
		We assume that $(u)_{S_1}=0$ and $(Du)_{S_1}=0$. Otherwise, we consider 
		\[\tilde{u}=u-(u)_{S_1}-(Du)_{S_1}\cdot (x-\bar{x}),\] where $\bar{x}=\fint_{S_1}x\,\mathrm{d}x$, which solves the same equation as $u$. Since $p>n$, Morrey's inequality gives \[\Vert u\Vert_{L^{\infty}(S_1)}\le C\left(\mean{S_1}|Du|^p\,\mathrm{d}x\right)^{1/p}=CE_p(S_1)\le C=:M.\]
		Let $w$ be the solution of 
		\[\begin{cases}
			\det D^2w=1 & \text{in }S_1,\\
			w=0& \text{on }\partial S_1.
			\end{cases}\]
		Then Lemma \ref{lem: approx} implies that for any $\varepsilon\in(0,M)$, there exist $\delta_1=\delta_1(\varepsilon,n,q)\in (0,1)$ and $v$ solving \[\begin{cases}
			W^{ij}D_{ij}v=0 & \text{in }S_{1/2},\\
			v=u & \text{on } \partial S_{1/2},
		\end{cases}\] 
		 such that if \[\begin{cases}
		1-\delta_1\le \det D^2\phi\le 1+\delta_1 &\text{in }S_1,\\
		\phi=0 & \text{on }\partial S_1,
		\end{cases}\] 
		and 
		\[\left(\mean{S_1}|f(x)|^q\,\mathrm{d}x\right)^{1/q}\le \delta_1,\] then \[\Vert u-v\Vert_{L^{\infty}(S_{1/4})}<\varepsilon.\] 
		Then \[\Vert v\Vert_{L^{\infty}(S_{1/4})}\le \Vert u\Vert_{L^{\infty}(S_1)}+\Vert u-v\Vert_{L^{\infty}(S_{1/4})}\le M+\varepsilon\le 2M.\] Applying the $C^{1,1}$-regularity estimate in Lemma \ref{lem: c1alpha for hom data}, we have \[\Vert v\Vert_{C^{1,1}(S_{1/8})}\le C\Vert v\Vert_{L^{\infty}(S_{1/4})}\le CM.\] It follows that there exists an affine function $l=v(0)+Dv(0)\cdot x$ such that \[\Vert v-l\Vert_{L^{\infty}(B_r(0))}\le CMr^{ 2 }\]
		for all small $r\le r_0$ such that $B_{r_0}(0)\subset S_{1/8}$.	
		By Lemma \ref{lem: geo of sections} we choose 
		\[\delta_2=\delta_2(\sigma)\le \tau_0 \sigma^2\] 
		such that $S_{2\sigma}\subset B_{c\sqrt{\sigma}}$ for all $\sigma\le\sigma_0$. Then for all $\sigma\le\sigma_0$ with  $c\sqrt{\sigma_0}= r_0$, we have \[\Vert v-l\Vert_{L^{\infty}(S_{2\sigma})}\le\Vert v-l\Vert_{L^{\infty}(B_{c\sqrt{\sigma}})}\le  CM\sigma.\]
		
		We claim that there exists $\delta_3\le \tau_0(2\sigma)^2$ such that if 
		\[|\det D^2\phi-1|<\delta_3  \text{\ in\ } S_1\] 
		and $\phi=0$ on $\partial S_1$, then \begin{equation}\label{eq: claim-1}
			\left(\mean{S_{\sigma}}|Du|^p\,\mathrm{d}x\right)^{1/p}\le \frac{C}{\sqrt{\sigma}}\Vert u\Vert_{L^{\infty}(S_{2\sigma})}+C\sqrt{\sigma}\left(\mean{S_{2\sigma}}|f|^q\,\mathrm{d}x\right)^{1/q},
		\end{equation} where $C=C(n,p,q)>0$. 
		Indeed, we consider \[\tilde{\phi}(y):=\frac{1}{2\sigma}\left[\phi((2\sigma)^{1/2}A^{-1}y)-\phi(0)-2\sigma\right] \]
for all $y\in (2\sigma)^{-1/2}AS_{2\sigma}(\phi)=S_1(\tilde{\phi})$, where $A$ is from Lemma \ref{lem: geo of sections} satisfying $\det A=1$, $\Vert A^{-1}\Vert \le c_1^{-1/2}$ and $\Vert A\Vert\le c_2^{1/2}$ for two universal constants. Denote 
\[\tilde{u}(y)=u((2\sigma)^{1/2}A^{-1}y)\quad \text{and}\quad \tilde{f}(y)=2\sigma f((2\sigma)^{1/2}A^{-1}y).\] 
Applying Lemma \ref{lem: W1p-est} to $\tilde{u}$ in $S_{1}(\tilde{\phi})$, we obtain that \[\left(\mean{S_{1/2}(\tilde{\phi})}|D\tilde{u}|^p\,\mathrm{d}y\right)^{1/p}\le C\Vert \tilde{u}\Vert_{L^{\infty}(S_1(\tilde{\phi}))}+C \left(\mean{S_1(\tilde{\phi})}|\tilde{f}|^{q}\,\mathrm{d}y\right)^{1/q}.\]	Then we scale it back to obtain \eqref{eq: claim-1}.
		
		Since $u-l$ solves the same equation as $u$, we may apply \eqref{eq: claim-1} to $u-l$ to obtain \[\left(\mean{S_{\sigma}}|D(u-l)|^p\,\mathrm{d}x\right)^{1/p}\le \frac{C}{\sqrt{\sigma}}\Vert u-l\Vert_{L^{\infty}(S_{2\sigma})}+C\sqrt{\sigma}\left(\mean{S_{2\sigma}}|f|^{q}\,\mathrm{d}x\right)^{1/q}.\] Then \begin{align*}
		E_p(S_{\sigma})&=\left(\mean{S_{\sigma}}|Du-(Du)_{S_{\sigma}}|^p\,\mathrm{d}x\right)^{1/p}\le 2\left(\mean{S_{\sigma}}|D(u-l)|^p\,\mathrm{d}x\right)^{1/p}\\
		&\le \frac{C}{\sqrt{\sigma}}\left(\Vert u-v\Vert_{L^{\infty}(S_{2\sigma})}+\Vert v-l\Vert_{L^{\infty}(S_{2\sigma})}\right)+C\sigma^{\frac{1}{2}-\frac{n}{2q}}\left(\mean{S_1}|f|^{q}\,\mathrm{d}x\right)^{1/q}\\
		&\le C\left(\varepsilon\sigma^{-1/2}+M\sigma^{1/2}+\sigma^{\frac{q-n}{2q}}\delta\right).
		\end{align*}
		 We first choose $\sigma$ sufficiently small such that $CM\sigma^{1/2}\le c_0/3$, then choose $\varepsilon$ sufficiently  small such that $C\varepsilon\sigma^{-1/2}\le c_0/3$. Finally we choose $\delta\le \min\{\delta_1,\delta_2,\delta_3\}$ small enough such that $C\delta\sigma^{\frac{q-n}{2q}}\le c_0/3$. In conclusion, we obtain $E_p(S_{\sigma}) \le c_0$.
	\end{proof}
	\begin{thm}\label{thm: Du potential est}
		Let $\Omega$ be a normalized convex domain in $\mathbb{R}^n$. Let $\phi\in C^2(\Omega)$ be a strictly convex function in $\Omega$ with $\phi=0$ on $\partial\Omega$. Let $x_0$ be the minimum point of $\phi$ and a Lebesgue point of $Du$. Let $f\in L^q(\Omega)$ with $\frac{n}{2}<q\le n$ and let $u$ be a solution of \eqref{eq: LMA equation}. Choose $p$ as \eqref{eq: p-choice} in Theorem~\ref{thm: gradient potential est}. Then for any $0<\theta<1$ and $\frac{1}{2}<\alpha<1$, there exist $C>0$, $\delta\in (0,1)$ and $h_0>0$, all depending only on $n$, $\alpha$, $\theta$, $p$ and $q$ such that if 
		\[1-\delta\le \det D^2\phi\le 1+\delta \ \ \text{in\ }  \Omega,\]
		then \begin{equation}\label{eq: Du potential-est}
			\begin{aligned}
				|Du(x_0)|\le& C\mean{S_{\phi}(x_0,h)}|Du(x)|\,\mathrm{d}x+ C\left(\mean{S_{\phi}(x_0,h^{\theta})}|Du(x)|^p\,\mathrm{d}x\right)^{1/p}\\[5pt]
				&\qquad+C\mathbf{I}_{\phi,q}^{f,\alpha}(x_0,h^{\theta})
			\end{aligned}
		\end{equation}
		whenever $S_{\phi}(x_0,h^{\theta})\Subset\Omega$ with $h\le h_0$.
	\end{thm}
	\begin{proof}
		For convenience, we assume that the minimum point of $\phi$ is $x_0=0$. 
		Let $\delta $ and $\sigma$ be as in Lemma \ref{lem: decay-est}.
		
		Now we claim that there exists a sequence of invertible matrices $\{A_j\}_{j=0}^{\infty}$ with $A_0=I$ and $\det A_j=1$ such that for all $j=1,2,3,\cdots$,
		\begin{eqnarray*}
		&&\Vert A_{j-1}A_j^{-1}\Vert\le \frac{1}{\sqrt{c_1}},\\ 
		&&\Vert A_j\Vert\le \sqrt{c_2(1+Ca_0)(1+Ca_1)\cdots (1+Ca_{j-1})},\\
		&&\Vert A_j^{-1}\Vert\le 1/\sqrt{c_1(1-Ca_1)\cdots (1-Ca_{j-1})},\\
		&&B_{(1-a_j)\sqrt{2}}(0)\subset \sigma^{-\frac{j}{2}}A_jS_{\sigma^j}(\phi)\subset B_{(1+a_j)\sqrt{2}}(0),
		\end{eqnarray*}
		where $C,c_1,c_2$ are the same constants as in Lemma \ref{lem: geo of sections} and the sequence $\{a_j\}_{j=0}^{\infty}$ is defined by
		\begin{align*}
			& a_0:=0, \quad a_1:=C\left(\sigma^{1/2}+\sigma^{-1}\delta^{1/2}\right),\\
			&a_j:=C\left(a_{j-1}\sigma^{1/2}+\sigma^{-1}\delta^{1/2}\right)\quad \text{for all }j\ge 2.
		\end{align*} The claim above is proved by induction similarly to \cite[Theorem 4.5]{GN1}. For convenience, we indicate the proof for the cases $j=1,2$.

\vskip 8pt
	
\underline{$j=1$}: Applying Lemma \ref{lem: geo of sections}, we obtain a positive definite matrix $M=A^tA$ with $\det M=1$, $c_1I\le M\le c_2I$ such that if we take $A_1:=A$ then 
\[B_{(1-a_1)\sqrt{2}}(0)\subset \sigma^{-\frac{1}{2}}A_1S_{\sigma}(\phi)\subset B_{(1+a_1)\sqrt{2}}(0)\]
with 
\[a_1:=C(\sigma^{1/2}+\sigma^{-1}\delta^{1/2})<1-\frac{6}{5\sqrt{2}}<\frac{1}{4}.\] Since $\det M=1$, $c_1I\le M\le c_2I$, we have $\Vert A_1^{-1}\Vert\le 1/\sqrt{c_1}$ and $\Vert A_1\Vert\le \sqrt{c_2}$. 
\vskip 8pt

\underline{$j=2$}: Let $\Omega_1:=\sigma^{-\frac{1}{2}}A_1S_{\sigma}(\phi)$ and \[\phi_1(y):=\frac{1}{\sigma}\left(\phi(\sigma^{\frac{1}{2}}A_1^{-1}y)-\phi(0)-\sigma\right), \quad \forall y\in \Omega_1.\] Then we have $\Omega_1=S_1(\phi_1)$ and  
\[
			\begin{cases}
				1-\delta\le \det D^2\phi_1\le 1+\delta & \text{in }\Omega_1\\
				\phi_1=0 & \text{on }\partial\Omega_1.
			\end{cases}\] Applying Lemma \ref{lem: geo-2} for $\phi_1$ and $\Omega_1$, we obtain a positive definite matrix $M=A^tA$ satisfying $\det M=1$ and $(1-Ca_1)I\le M\le (1+Ca_1)I$, such that 
\[B_{(1-a_2)\sqrt{2}}(0)\subset \sigma^{-\frac{1}{2}}AS_{\sigma}(\phi_1)\subset B_{(1+a_2)\sqrt{2}}(0),\]
with 
\[a_2:=C(a_1\sigma^{1/2}+\sigma^{-1}\delta^{1/2}).\] 
Define  $A_2:=AA_1$. Then $\det A_2=1$. As $S_{\sigma}(\phi_1)=\sigma^{-\frac{1}{2}}A_1S_{\sigma^2}(\phi)$, we conclude that \[B_{(1-a_2)\sqrt{2}}(0)\subset \sigma^{-1}A_2S_{\sigma^2}(\phi)\subset B_{(1+a_2)\sqrt{2}}(0).\] 
Since $(1-Ca_1)I\le M\le (1+Ca_1)I$ and $A_1A_2^{-1}=A^{-1}$, we have 
\[\Vert A_1A_2^{-1}\Vert =\Vert A^{-1}\Vert\le \frac{1}{\sqrt{1-Ca_1}}\le \frac{1}{\sqrt{c_1}}.\] 
By $A^tA=M$, we have 
\[(1-Ca_1)|x|^2\le |Ax|^2\le (1+Ca_1)|x|^2.\] 
Then \[|A_2x|^2=|AA_1x|^2\le (1+Ca_1)|A_1x|^2\le c_2(1+Ca_1)|x|^2,\] which yields $\Vert A_2\Vert\le \sqrt{c_2(1+Ca_1)}$. 
		Similarly, we have $\Vert A^{-1}_2\Vert\le \frac{1}{\sqrt{c_1(1-Ca_1)}}$.
		
		Now we estimate $\Vert A_j\Vert$ and $\Vert A_j^{-1}\Vert$. Choose $\sigma$ so small that $C\sqrt\sigma<1/2$ and then choose $\delta$ so small that $Ca_j<1/2$ for all $j$. The recursion for $a_j$ gives, for all $j\ge1$, \[a_j\le (C\sqrt{\sigma})^j+\frac{2C\sqrt{\delta}}{\sigma}.\]
		It follows that for all $j\ge 2$, 
		\begin{align*}
			\prod_{i=1}^{j-1}(1+Ca_i)&=\exp\left(\sum_{i=1}^{j-1}\log (1+Ca_i)\right)\le \exp\left(C\sum_{i=1}^{j-1}a_i\right)\\
			&\le  \exp\left(C\sum_{i=1}^{j-1} (C\sqrt{\sigma})^i+\frac{2(j-1)C\sqrt{\delta}}{\sigma}\right)\le C\exp\left(\frac{2jC\sqrt{\delta}}{\sigma}\right).
		\end{align*}
		Then we have
		 \begin{equation*}
		\Vert A_j\Vert \le \left[c_2\prod_{i=1}^{j-1}(1+C a_i)\right]^{\frac{1}{2}} \le  
		C\exp\left(\frac{jC\sqrt{\delta}}{\sigma}\right),
		\end{equation*}
		while $\Vert A_j^{-1}\Vert$ satisfies a similar estimate. So we have \begin{equation}\label{eq: A_k-est}
			\Vert A_j\Vert+\Vert A_j^{-1}\Vert\le CL^j,
		\end{equation} where \[L=\exp(2C\sqrt{\delta}/\sigma).\]
		

		Now we prove \eqref{eq: Du potential-est}. For convenience, we denote $h_j:=\sigma^j$ and $S_j:=S_{h_j}(\phi)$ for all $j\ge 1$. Let $\Omega_j:=h_j^{-1/2}A_jS_{j}$. Consider
		 \[\phi_j(y):=\frac{1}{h_j}\left(\phi(h_j^{1/2}A_j^{-1}y)-\phi(0)-h_j\right), \quad u_j(y):=\frac{u(h_j^{1/2}A_j^{-1}y)}{M_j},\quad \forall y\in \Omega_j,\] 
		where 
		\begin{eqnarray*}
		M_j&:=&h_j^{1/2} \left(\mean{S_{j}}|Du(x)A_j^{-1}-(Du\,A_j^{-1})_{S_{j}(\phi)}|^p\,\mathrm{d}x\right)^{1/p}\\
		&&+\frac{h_j}{\delta}\left(\mean{S_{j}}|f(x)|^q\,\mathrm{d}x\right)^{1/q}+\varepsilon_0,\end{eqnarray*} 
		for  a small parameter $\varepsilon_0$. Then for $\Omega_j=S_{1}(\phi_j)$, we have
		\[\begin{cases}
				1-\delta\le \det D^2\phi_j\le 1+\delta & \text{in }\Omega_j,\\
				\phi_j=0 & \text{on }\partial\Omega_j,
			\end{cases}
		\] and \[\operatorname{tr}(\Phi_j D^2u_j)=\frac{h_jf(h_j^{1/2}A_j^{-1}y)}{M_j}=:f_j(y)\]
		in $\Omega_j$, where $\Phi_j:=\det D^2\phi_j(y)(D^2\phi_j)^{-1}(y)$.
		
		 Noting that 
		 \begin{align*}
		E_p(S_1(\phi_j))&:=\left(\mean{S_1(\phi_j)}|Du_j-(Du_j)_{S_1(\phi_j)}|^p\,\mathrm{d}y\right)^{1/p}\\
		&=\frac{h_j^{1/2}}{M_j}\left(\mean{S_j}|Du(x)A_j^{-1}-(Du\,A_j^{-1})_{S_{j}}|^p\,\mathrm{d}x\right)^{1/p}\le 1
		\end{align*}and \[	\left(\mean{S_1(\phi_j)}|f_j(y)|^q\,\mathrm{d}y\right)^{1/q}=\frac{h_j}{M_j}\left(\mean{S_{j}}|f(x)|^q\,\mathrm{d}x\right)^{1/q}\le \delta,\]
		we can apply Lemma \ref{lem: decay-est} to obtain
		\begin{align*}
		c_0&\ge 	E_p(S_{\sigma}(\phi_j)):=\left(\mean{S_{\sigma}(\phi_j)}|Du_j-(Du_j)_{S_{\sigma}(\phi_j)}|^p\,\mathrm{d}y\right)^{1/p}\\
			&= \frac{h_j^{1/2}}{ M_j}\left(\mean{S_{j+1}(\phi)}|Du(x) A_j^{-1}-(Du\,A_j^{-1})_{S_{j+1}(\phi)}|^p\,\mathrm{d}x\right)^{1/p}.
		\end{align*}
		For convenience, we write \[\mathcal{E}_j:=\left(\mean{S_{j}}|Du(x)A_j^{-1}-(Du\,A_j^{-1})_{S_{j}}|^p\,\mathrm{d}x\right)^{1/p}.\]
		Letting  $\varepsilon_0\to0$, we have \begin{align}
			\mathcal{E}_{j+1}&\le \Vert A_jA_{j+1}^{-1}\Vert \left(\mean{S_{{j+1}}}|Du(x) A_j^{-1}-(Du\,A_j^{-1})_{S_{j+1}}|^p\,\mathrm{d}x\right)^{1/p} \nonumber\\
			&\le \frac{c_0}{\sqrt{c_1}}h_j^{-1/2}M_j
			= \frac{c_0}{\sqrt{c_1}}\mathcal{E}_j+\frac{c_0\sqrt{h_j}}{\sqrt{c_1}\delta}\left(\mean{S_{j}}|f(x)|^q\,\mathrm{d}x\right)^{1/q}.\label{eq: E_p-est}
		\end{align}
		
		Given $0<h\ll 1$, there exists $k\ge 2$ such that \[h_{k}<h\le h_{k-1},\] so that $S_k\subset S_{\phi}(x_0,h)\subset S_{k-1}$ and the two sections have comparable volume. Given $0<\theta<1$, let \[N=\left \lceil \theta k \right \rceil+1, \] which makes sure \[\sigma^{2+\theta}h^{\theta}\le h_N\le h_{N-1}\le h^{\theta}.\] If we denote \[E_j:=\left(\mean{S_j}|Du-(Du)_{S_j}|^p\,\mathrm{d}x\right)^{1/p},\quad\rho:=\frac{c_0}{\sqrt{c_1}}\quad\text{ and }\quad F_j:=\left(\mean{S_j}|f(x)|^q\,\mathrm{d}x\right)^{1/q}.\] 
        then \[E_j\le \Vert A_j\Vert\mathcal{E}_j,\quad \quad\mathcal{E}_j\le \Vert A_j^{-1}\Vert E_j.\]
        We choose the constants in the following order. First choose $c_0$ small in Lemma~\ref{lem: decay-est}; this fixes a corresponding $\sigma$. Then shrink $\delta$ so that $L=\exp(2C\sqrt\delta/\sigma)$ is sufficiently close to $1$. With this choice we have
        \begin{equation}\label{eq: choice of sigma and delta}
			L^{1+\theta}\rho^{1-\theta}<1,\quad L\sigma^{(1-\theta)(\alpha-\frac{1}{2})}\le 1,\quad L\rho<1.
		\end{equation} Iterating \eqref{eq: E_p-est} from $N$ to $i$, we have \begin{equation}\label{eq: E_i-est}
			E_i\le CL^i\mathcal{E}_i \le CL^i\left(\rho^{i-N}\mathcal{E}_N+\sum_{j=N}^{i-1}\rho^{i-1-j}h_j^{1/2}F_j\right).
		\end{equation}Then \begin{align*}
		\sum_{i=k}^{\infty}E_i&\le C\sum_{i=k}^{\infty}L^{i+N}\rho^{i-N}E_N+C\sum_{i=k}^{\infty}\sum_{j=N}^{i-1}L^i\rho^{i-1-j}h_j^{1/2}F_j\\
		&\le CL^{N+k}\rho^{k-N}E_N+C\sum_{j=N}^{k-1}L^k\rho^{k-1-j}h_j^{1/2}F_j+C\sum_{j=k}^{\infty}L^{j+1}h_j^{1/2}F_j.
		\end{align*}
		 The choice of $c_0$, $\delta$ and $\sigma$ in \eqref{eq: choice of sigma and delta} ensures that \begin{align*}
		L^{N+k}\rho^{k-N}&\le C(L^{1+\theta}\rho^{1-\theta})^k\le C,\\[5pt]
		L^{k}\rho^{k-1-j}h_j^{1/2}&=L^k\rho^{k-N-1}(\rho^{-1}\sigma^{\alpha-\frac{1}{2}})^{j-N}h_N^{\alpha-\frac{1}{2}}h_j^{1-\alpha}\\
		&\le C\max\{(L\rho^{1-\theta})^k, (L\sigma^{(1-\theta)(\alpha-\frac{1}{2})})^k\} h_N^{\alpha-\frac{1}{2}}h_j^{1-\alpha}\\
		&\le  Ch_{N}^{\alpha-\frac{1}{2}}h_j^{1-\alpha}, \quad\quad\quad\quad  \text{for all }N\le j<k,\\[5pt]
		L^{j+1}h_j^{1/2}&= L^{j+1}\sigma^{(j-N)(\alpha-\frac{1}{2})}h_N^{\alpha-\frac{1}{2}}h_j^{1-\alpha}
		\le C(L\sigma^{(1-\theta)(\alpha-\frac{1}{2})})^j h_N^{\alpha-\frac{1}{2}}h_j^{1-\alpha} \\
		&\le Ch_N^{\alpha-\frac{1}{2}}h_j^{1-\alpha},\quad \quad\quad\quad\quad \text{for all }j\ge k.
		\end{align*} Therefore, \begin{equation}\label{eq: sum-E_i-est}
		\begin{split}
			\sum_{i=k}^{\infty}E_i&\le C E_N+Ch_{N}^{\alpha-\frac{1}{2}}\sum_{j=N}^{\infty}h_j^{1-\alpha}F_j\\
		&\le CE_N+Ch_{N}^{\alpha-\frac{1}{2}}\int_0^{h_{N-1}}\left(\mean{S_{\phi}(x_0,\rho)}|f(x)|^q\,\mathrm{d}x\right)^{1/q}\frac{\mathrm{d}\rho}{\rho^{\alpha}}\\
		&\le C\left(\mean{S_{\phi}(x_0,h^{\theta})}|Du(x)|^p\,\mathrm{d}x\right)^{1/p}+C\mathbf{I}_{\phi,q}^{f,\alpha}(x_0,h^{\theta}).
		\end{split}
		\end{equation}
		Hence, \begin{align*}
			&|Du(x_0)|=\lim\limits_{m\to\infty}|(Du)_{S_{m+1}}|\le\sum_{i=k}^{\infty}|(Du)_{S_{i+1}}-(Du)_{S_i}|+|(Du)_{S_k}|\\
			&\,\,\le \sum_{i=k}^{\infty}\left(\mean{S_{i+1}}|Du-(Du)_{S_i}|^p\,\mathrm{d}x\right)^{1/p}+|(Du)_{S_k}|\\
			&\,\,\le C\sum_{i=k}^{\infty}E_i+C\mean{S_{\phi}(x_0,h)}|Du(x)|\,\mathrm{d}x\\
			&\,\,\le C\mean{S_{\phi}(x_0,h)}|Du(x)|\,\mathrm{d}x+ C\left(\mean{S_{\phi}(x_0,h^{\theta})}|Du(x)|^p\,\mathrm{d}x\right)^{1/p}+C\mathbf{I}_{\phi,q}^{f,\alpha}(x_0,h^{\theta}).
		\end{align*}
	\end{proof}
	
	\begin{proof}[Proof of Theorem \ref{thm: gradient potential est}]
	Let $x_0\in\Omega'\Subset\Omega$. By \eqref{eq: continuous g} and \cite[Lemma 2.1]{GN1}, there exist $C_1$, $C_2$, $h_2$ and $b>0$ such that \[B(x_0,C_1h)\subset S_{\phi}(x_0,h)\subset B(x_0,C_2h^b)\qquad \text{for all }0<h\le h_2,\] 
	where $b=b(n,\lambda,\Lambda)$ and $C_1$, $C_2$, $h_2$ depend only on $n$, $\lambda$, $\Lambda$ and $\operatorname{dist}(\Omega',\partial\Omega)$. If necessary, decrease $h_2$ such that these sections are compactly contained in a fixed interior neighborhood of $\Omega'$. 
	Using the modulus of continuity of $g$, choose $h_1\in (0,h_2]$ such that \[|g(y)-g(x_0)|\le \lambda\delta\qquad \text{for all }y\in S_{\phi}(x_0,h_1), x_0\in\Omega',\] 
	where $\delta$ is the small constant in Theorem~\ref{thm: Du potential est}. 
	By John's lemma, there exists an affine transformation $T:=A(x-x_0)+y_0$ such that 
	\[B_{1}(0)\subset TS_{\phi}(x_0,h_1)\subset B_n(0),\] 
	which implies \begin{equation}\label{eq: ||A||-est}
		\Vert A\Vert\le Ch_1^{-1},\quad \Vert A^{-1}\Vert\le Ch_1^b,
	\end{equation}
	where $C$ depends only on $n$, $\lambda$, $\Lambda$ and $\operatorname{dist}(\Omega',\partial\Omega)$.
	By the volume estimates of sections (see \cite[Lemma 5.6(i)]{Le5}), we have\begin{equation}\label{eq: hdet^{2/n}=1}
		C^{-1}\le |\det A|^{\frac{2}{n}}h_1\le C
	\end{equation} for some $C>0$ depending only on $n,\lambda,\Lambda$, which means that $|\det A|\approx h_1^{-n/2}$.
	
	Denote $\widetilde{\Omega}:=TS_{\phi}(x_0,h_1)$ and consider the functions
	\begin{align}
		\tilde{\phi}(y)&:=K\left[\phi(T^{-1}y)-\phi(x_0)-D\phi(x_0)\cdot (T^{-1}y-x_0)-h_1\right],\label{eq: scale phi}\\
		\tilde{u}(y)&:=Kg(x_0)u(T^{-1}y),\label{eq: scale u}
	\end{align} 
for $y\in\widetilde{\Omega}$, where $K=\frac{|\det A|^{2/n}}{g(x_0)^{1/n}}$.
Noting that 
\begin{equation*}
		\det D^2\tilde{\phi}(y)=\frac{g(T^{-1}y)}{g(x_0)}, \quad \text{and }\quad |g(T^{-1}y)-g(x_0)|<\lambda\delta\,\,\text{in }\widetilde{\Omega},
	\end{equation*} we have \[1-\delta\le\det D^2\tilde{\phi}(y)\le 1+\delta.\]
	Meanwhile,  $\tilde{\phi}=0$ on $\partial \widetilde{\Omega}$ and \[\tilde{\Phi}^{ij}D_{ij}\tilde{u}=\Phi^{ij}(T^{-1}y)D_{ij}u(T^{-1}y)=\tilde{f}\] for any $y\in \widetilde{\Omega}=S_{\tilde{\phi}}(y_0,Kh_1)$, where $\tilde{f}$ is defined by $\tilde{f}(y)= f(T^{-1}y)$ for any $y\in \widetilde{\Omega}$. 
	
	Now we apply Theorem \ref{thm: Du potential est} and scale back to obtain \begin{align*}
		&|Du(x_0)|\le \frac{\Vert A\Vert}{Kg(x_0)}|D\tilde{u}(y_0)|\le C|D\tilde{u}(y_0)|\\
		&\,\,\le C\mean{S_{\tilde{\phi}}(y_0,h)}|D\tilde{u}(y)|\,\mathrm{d}y+ C\left(\mean{S_{\tilde{\phi}}(y_0,h^{\theta})}|D\tilde{u}(y)|^p\,\mathrm{d}y\right)^{1/p}+C\mathbf{I}_{\tilde{\phi},q}^{\tilde{f},\alpha}(y_0,h^{\theta})\\
		&\,\,\le C \mean{S_{\phi}(x_0,t)}|Du(x)|\,\mathrm{d}x+  C \left(\mean{S_{\phi}(x_0,t^{\theta})}|Du(x)|^p\,\mathrm{d}x\right)^{1/p}+ C\mathbf{I}_{\phi,q}^{f,\alpha}(x_0,ct^{\theta}),
	\end{align*} for all $0<t\le h_{\ast}$, where $h_{\ast}>0$ is chosen uniformly so that \[h_{\ast}< \min\{K^{-1}h_0, K^{1/\theta-1}h_1^{1/\theta}\},\]
	and the constants $C$, $c>0$ depend only on $n$, $\alpha$, $\theta$, $\lambda$, $\Lambda$, $p$, $q$, $\operatorname{dist}(\Omega',\partial\Omega)$ and the modulus of continuity of $g$.
\end{proof}
		
		\section{Interior gradient continuity via potentials}\label{sec: gradient continuity}
		
	We first consider the case where $\det D^2\phi$ is sufficiently
	close to one.
		\begin{thm}\label{thm: modulus est}
		Let $\Omega$ be a normalized convex domain in $\mathbb{R}^n$, and let $\phi\in C^2(\Omega)\cap C(\overline{\Omega})$ be strictly convex with $\phi=0$ on $\partial\Omega$. Let $f\in L^q(\Omega)$ with $\frac{n}{2}<q\le n$, and let $u$ be the solution to \eqref{eq: LMA equation}. Choose $p$ as \eqref{eq: p-choice} in Theorem \ref{thm: gradient potential est}. Fix $\frac{1}{2}<\alpha<1$ and $0<\theta<1$. Then there exists $\delta\in (0,1/2)$, depending only on $n$, $p$, $q$, $\alpha$, $\theta$, such that if 
		\[1-\delta\le \det D^2\phi\le 1+\delta \:\text{\ in\ }  \Omega,\]
		and \begin{equation}\label{eq: I-to0}
		\lim\limits_{h\to 0}\sup_{\substack{x\in\Omega'\\S_{\phi}(x,h)\Subset\Omega}}\mathbf{I}_{\phi,q}^{f,\alpha}(x,h)= 0 \,\,\: \text{for every }\Omega'\Subset\Omega,
	\end{equation}  then $Du$ has a continuous representative. Moreover, we have \begin{equation}\label{eq: modulus est}
		|Du(x_1)-Du(x_2)|\le C\Vert Du\Vert_{L^{\infty}(\Omega'')}|x_1-x_2|^{\kappa\beta} +C\max_{x\in \{x_1,x_2\}}\mathbf{I}_{\phi,q}^{f,\alpha}(x,c|x_1-x_2|^{\beta\theta}) 
		\end{equation} for all $x_1,x_2\in\Omega'\Subset\Omega''\Subset\Omega$, where $C$, $c>0$ depend only on $n$, $p$, $q$, $\alpha$, $\theta$, $\operatorname{dist}(\Omega',\partial\Omega'')$ and $\operatorname{dist}(\Omega'',\partial\Omega)$. The exponents $\kappa>0$ and $\beta>1$ depend only on $n$, $p$, $q$, $\alpha$, $\theta$ and $n$ respectively.
		\end{thm}
		\begin{proof}
			Let $\Omega'\Subset\Omega''\Subset\Omega$. We choose $\delta$ sufficiently small so that Theorem~\ref{thm: Du potential est} applies after normalization. By Theorem \ref{thm: gradient potential est} and \eqref{eq: I-to0}, we have 
			\begin{equation}\label{eq: Du-bdd}
				\Vert Du\Vert_{L^{\infty}(\Omega'')}< \infty.
			\end{equation}
			For each $x_0\in\Omega'$, normalize a section $S_{\phi}(x_0,h_1)\Subset\Omega''$ as in \eqref{eq: scale phi} \eqref{eq: scale u}, where the height $h_1$ is uniform by \cite[Lemma 2.1]{GN1}. We use the notation from the proof of Theorem \ref{thm: Du potential est} in the normalized coordinates, with the corresponding heights $h_j=K^{-1}\sigma^j$ and sections $S_j=S_{\phi}(x_0,h_j)$.
		
			We first show that $\{(Du)_{S_{\phi}(x_0,t)}\}_{t\le h_0}$ is a Cauchy family uniformly in $x_0\in\Omega'$. In this paragraph the decay factor is denoted by $\rho=c_0/\sqrt{c_1}$, while $r,t$ denote section heights. Thus, taking $h_k\le h<h_{k-1}$, and $N=\left \lceil \theta k \right \rceil+1$,
			and arguing similarly to \eqref{eq: sum-E_i-est}, for $m>k>0$, we have \begin{align*}
				|(Du)_{S_{m+1}}-(Du)_{S_{k-1}}|&\le \sum_{i=k-1}^{m}|(Du)_{S_{i+1}}-(Du)_{S_i}|\le C\sum_{i=k-1}^{\infty} E_i\\
				&\le C(L^{1+\theta}\rho^{1-\theta})^k E_N+C\mathbf{I}_{\phi,q}^{f,\alpha}(x_0,ch^{\theta})\\
				&\le Ch_k^{\kappa} \Vert Du\Vert_{L^{\infty}(\Omega'')} + C\mathbf{I}_{\phi,q}^{f,\alpha}(x_0,ch^{\theta}) ,
			\end{align*} where \[\kappa:=-\frac{\log (L^{1+\theta}\rho^{1-\theta})}{\log(1/\sigma)}>0.\] Consequently, for any $0<r<t\le h_0$, \[|(Du)_{S_{\phi}(x_0,r)}-(Du)_{S_{\phi}(x_0,t)}|\le Ct^{\kappa}\Vert Du\Vert_{L^{\infty}(\Omega'')}+C\mathbf{I}_{\phi,q}^{f,\alpha}(x_0,ct^{\theta}).\] Using \eqref{eq: I-to0}, the averages converge uniformly in $x_0\in\Omega'$. Denote the uniform limit by \[V(x_0):=\lim\limits_{r\to 0} (Du)_{S_{\phi}(x_0,r)}.\] Then \begin{equation}\label{eq: Du-est}
			|V(x_0)-(Du)_{S_{\phi}(x_0,t)}|\le Ct^{\kappa}\Vert Du\Vert_{L^{\infty}(\Omega'')}+C\mathbf{I}_{\phi,q}^{f,\alpha}(x_0,ct^{\theta}).
			\end{equation} 
			 For each fixed $t>0$, Lemma \ref{lem: average contin} implies that \[x_0\mapsto (Du)_{S_{\phi}(x_0,t)}\] is continuous. The convergence in \eqref{eq: Du-est} is uniform, hence $V$ is continuous. Moreover, $V=Du$ at every Lebesgue point of $Du$, and we henceforth use $V$ as the continuous representative of $Du$ in $\Omega'$.
			
				Finally, it remains to estimate the modulus of continuity of $Du$. We first consider $x_1,x_2\in\Omega'$ such that $x_2\in \partial S_{\phi}(x_1,\tau)$ with $\tau\le h_0$. The engulfing property of sections (see \cite[Theorem 5.28]{Le5}) yields $S_{\phi}(x_1,\tau)\subset S_{\phi}(x_2,\theta_0\tau)$ for some $\theta_0=\theta_0(n,p,q,\alpha)>0$. So we have \begin{align*}
					|Du(x_1)-Du(x_2)|&\le |Du(x_1)-(Du)_{S_{\phi}(x_1,\tau)}|+|Du(x_2)-(Du)_{S_{\phi}(x_2,\theta_0\tau)}|\\
					&\quad\quad\quad +|(Du)_{S_{\phi}(x_1,\tau)}-(Du)_{S_{\phi}(x_2,\theta_0\tau)}|=:\text{I}+\text{II}+\text{III}.
				\end{align*} 
				By \eqref{eq: Du-est}, we have \[\text{I}+\text{II}\le  C\tau^{\kappa}\Vert Du\Vert_{L^{\infty}(\Omega'')}+C\mathbf{I}_{\phi,q}^{f,\alpha}(x_1,c\tau^{\theta})+C\mathbf{I}_{\phi,q}^{f,\alpha}(x_2,c(\theta_0\tau)^{\theta}). \] The third term \begin{align*}
					\text{III}&\le \mean{S_{\phi}(x_1,\tau)}|Du(y)-(Du)_{S_{\phi}(x_2,\theta_0\tau)}|\,\mathrm{d}y\\
					&\le C\mean{S_{\phi}(x_2,\theta_0\tau)}|Du(y)-(Du)_{S_{\phi}(x_2,\theta_0\tau)}|\,\mathrm{d}y\le CE_p(S_{\phi}(x_2,\theta_0\tau)).
				\end{align*}
				Since there exists $k\in\mathbb{N}$ such that $h_{k}< \theta_0\tau\le h_{k-1}$, similarly to \eqref{eq: sum-E_i-est} we obtain \[\text{III}\le C E_{k-1}\le C\sum_{i=k-1}^{\infty} E_i\le C\tau^{\kappa}\Vert Du\Vert_{L^{\infty}(\Omega'')}+C\mathbf{I}_{\phi,q}^{f,\alpha}(x_2, c(\theta_0\tau)^{\theta}).\]
				 Therefore, we obtain that \[|Du(x_1)-Du(x_2)|\le C\tau^{\kappa}\Vert Du\Vert_{L^{\infty}(\Omega'')}+C\mathbf{I}_{\phi,q}^{f,\alpha}(x_1, c\tau^{\theta})+C\mathbf{I}_{\phi,q}^{f,\alpha}(x_2, c\tau^{\theta}).\]
				By the $C^{1,\alpha}$ estimate for the Monge-Amp\`ere equation (see \cite[Theorem 5.18]{Le5}), we see that \[\tau=\phi(x_2)-\phi(x_1)-D\phi(x_1)\cdot(x_2-x_1)\le C|x_1-x_2|^{\beta},\] for some $\beta=\beta(n)>1$ and $C=C(n, \operatorname{dist}(\Omega',\partial\Omega))>0$. Hence, we have \[|Du(x_1)-Du(x_2)|\le C\Vert Du\Vert_{L^{\infty}(\Omega'')}|x_1-x_2|^{\kappa\beta}+C\max_{x\in\{x_1,x_2\}}\mathbf{I}_{\phi,q}^{f,\alpha}(x, c|x_1-x_2|^{\theta\beta}).\]
				
				 If $x_2\notin S_{\phi}(x_1,\tau)$ for all $\tau\le h_0$, then we have \[h_0 \le \phi(x_2)-\phi(x_1)-D\phi(x_1)\cdot(x_2-x_1)\le C|x_1-x_2|^{\beta}.\] So we have \[|Du(x_1)-Du(x_2)|\le 2\sup_{\Omega''}|Du|\le Ch_0^{-\kappa}|x_1-x_2|^{\kappa\beta}\Vert Du\Vert_{L^{\infty}(\Omega'')}.\]
				Therefore, we obtain \eqref{eq: modulus est}.


		\end{proof}
		
		\begin{proof}[Proof of Theorem \ref{thm: grad-conti}]
			Using the same scaling in \eqref{eq: scale phi} and \eqref{eq: scale u} as in the proof of Theorem \ref{thm: gradient potential est}, one can obtain \begin{align*}
				|Du(x_1)-Du(x_2)|&\le C|D\tilde{u}(y_1)-D\tilde{u}(y_2)|\\
				&\le C\Vert Du\Vert_{L^{\infty}(\Omega'')}|y_1-y_2|^{\kappa\beta}+C\max_{y\in \{y_1,y_2\}}\mathbf{I}_{\tilde{\phi},q}^{\tilde{f},\alpha}(y, \tilde{c}|y_1-y_2|^{\theta\beta})\\
				&\le C\Vert Du\Vert_{L^{\infty}(\Omega'')}|x_1-x_2|^{\kappa\beta}+C\max_{x\in \{x_1,x_2\}}\mathbf{I}_{\phi,q}^{f,\alpha}(x, c|x_1-x_2|^{\theta\beta}),
			\end{align*} where the constants $C$ and $c$ have the dependence stated in Theorem~\ref{thm: grad-conti}.
		\end{proof}

		\begin{proof}[Proof of Corollary \ref{cor: borderline regularity}]
			We present the argument in a normalized section; the general local statement follows by the affine normalization employed in the proof of Theorem \ref{thm: gradient potential est}. All constants below are local constants; after normalization they are uniform on compact subsets. Fix a sufficiently small $\theta\in(0,1)$ and choose $p>n$ within the admissible range. The excess-decay iteration from the proof of Theorem \ref{thm: Du potential est} yields that, for every 
			$\Omega'\Subset\Omega''\Subset\Omega$, there exist constants $C,\kappa>0$ such that, for all $x\in\Omega'$ and all sufficiently small sections $S_{\phi}(x,h)\Subset\Omega''$, 
			\[E_p(S_{\phi}(x,h))\le Ch^{\kappa}\Vert Du\Vert_{L^p(\Omega'')}+C\omega_{\alpha,q}(h^{\theta},\Omega').\]  Since \[\operatorname{mosc}_{S_{\phi}(x,h)} Du\le CE_p(S_{\phi}(x,h)),\]
			it follows immediately that $Du\in VMO_{loc}(\Omega,\phi)$ whenever  
			$\omega_{\alpha,q}(r,\Omega')\to 0$ as $r\to 0$, and $Du\in BMO_{loc}(\Omega,\phi)$ whenever $\omega_{\alpha,q}$ remains locally bounded.

			  Assume now that \[M:=	\sup_{S_{\phi}(x_0,h)\Subset\Omega}h^{\frac{1-\varepsilon}{2}}\left(\mean{S_{\phi}(x_0,h)}|f(y)|^n\,\mathrm{d}y\right)^{1/n}<\infty\] for some $\varepsilon\in (0,1)$. Choose $1/2<\alpha<(1+\varepsilon)/2$. Then for all sufficiently small $h$, \[\mathbf{I}_{\phi,n}^{f,\alpha}(x,h)\le Mh^{\alpha-\frac{1}{2}}\int_0^h\rho^{-\frac{1-\varepsilon}{2}}\frac{\mathrm{d}\rho}{\rho^{\alpha}}\le CMh^{\varepsilon/2}.\] 
			  Theorem \ref{thm: grad-conti} therefore yields a local H\"older modulus for $Du$, hence $u\in C_{\text{loc}}^{1,\gamma}(\Omega)$ for some $\gamma>0$. Specifically, 
              using the modulus of continuity estimate \eqref{eq: Du-modulus} in Theorem \ref{thm: grad-conti} and  combining it with the estimate (3.4) in \cite{LN1}, we obtain \begin{align*}
                  |Du(y_1)-Du(y_2)|&\le C(\Vert Du\Vert_{L^{p}(\Omega'')}+M)|y_1-y_2|^{\min\{\kappa\beta, \beta\theta\varepsilon/2\}}\\
                  &  \le C(\Vert u\Vert_{L^{\infty}(\Omega)}+M)|y_1-y_2|^{\min\{\kappa\beta, \beta\theta\varepsilon/2\}},
              \end{align*} for any $y_1,y_2\in\Omega'$, which is consistent with the result in \cite{GN1}.
			  
			  Finally, if $f\in L^{s,\infty}(\Omega)$ with $s>n$, then for every $q\le n<s$, by the H\"older's inequality in Lorentz spaces \eqref{property 3: Holder ineq}, we obtain that \[\left(\int_{S_{\phi}(x,h)}|f|^q\,\mathrm{d}x\right)^{1/q}\le C\Vert f\Vert_{L^{s,\infty}(S_{\phi}(x,h))}|S_{\phi}(x,h)|^{\frac{1}{q}-\frac{1}{s}}\le C\Vert f\Vert_{L^{s,\infty}(S_{\phi}(x,h))}h^{\frac{n}{2q}-\frac{n}{2s}}.\] Taking $\alpha=1-\frac{n}{2s}>\frac{1}{2}$ gives a uniform bound for $\omega_{\alpha,q}$ and $Du\in BMO_{\text{loc}}(\Omega,\phi)$.   
              
              If $f\in L^{r}(\Omega)$ for some $r>n$, we set $\varepsilon:=1-n/r>0$. 
              By H\"older's inequality, we have 
              \begin{align*}
			  	h^{\frac{1-\varepsilon}{2}}\left(\mean{S_{\phi}(x,h)}|f(y)|^n\,\mathrm{d}y\right)^{1/n}\le Ch^{\frac{n-r}{2r}}\Vert f\Vert_{L^{n}(S_{\phi}(x,h))}\le C\Vert f\Vert_{L^r(\Omega)}<\infty,
			  \end{align*} which yields that $u$ is $C^{1,\gamma}$ locally.
		\end{proof}

        \vspace{1em}
	     \noindent\textbf{Acknowledgments.} The author would like to thank  his PhD supervisor, Prof. Bin Zhou, for his constant encouragement and helpful suggestions. The author would also like to thank Dr. Ling Wang for his interest and helpful comments. The author is also grateful to Prof. Nam Q. Le for his valuable comments and generous suggestions. The author would also like to thank the referees for their careful reading and valuable suggestions.

	\end{document}